\documentclass{amsart}
\usepackage[utf8]{inputenc}
\usepackage{amssymb,latexsym}
\usepackage{amsmath}
\usepackage{graphicx}
\usepackage{textcomp}
\usepackage[dvipsnames]{xcolor}

\usepackage{amsthm,amssymb,enumerate,graphicx, tikz}
\usepackage{amscd}
\usepackage{setspace}
\usepackage{comment}
\usepackage{hyperref}
\usepackage{cleveref}
\usepackage{caption}

\def\@logofont{\footnotesize}
\def\@setaddresses{\par
  \nobreak \begingroup
  \footnotesize
  \def\author##1{\nobreak\addvspace\bigskipamount}%
  \def\\{\par\nobreak}%
  \interlinepenalty\@M
  \def\address##1##2{\begingroup
    \par\addvspace\bigskipamount\indent
    \@ifnotempty{##1}{(\ignorespaces##1\unskip) }%
    {\scshape\ignorespaces##2}\par\endgroup}%
  \def\curraddr##1##2{\begingroup
    \@ifnotempty{##2}{\nobreak\indent\curraddrname
      \@ifnotempty{##1}{, \ignorespaces##1\unskip}\/:\space
      ##2\par}\endgroup}%
  \def\email##1##2{\begingroup
    \@ifnotempty{##2}{\nobreak\indent\emailaddrname
      \@ifnotempty{##1}{, \ignorespaces##1\unskip}\/:\space
      \ttfamily##2\par}\endgroup}%
  \def\urladdr##1##2{\begingroup
    \def~{\char`\~}%
    \@ifnotempty{##2}{\nobreak\indent\urladdrname
      \@ifnotempty{##1}{, \ignorespaces##1\unskip}\/:\space
      \ttfamily##2\par}\endgroup}%
  \addresses
  \endgroup
}
\renewcommand*\subjclass[2][2010]{%
  \def\@subjclass{#2}%
  \@ifundefined{subjclassname@#1}{%
    \ClassWarning{\@classname}{Unknown edition (#1) of Mathematics
      Subject Classification; using '2000'.}%
  }{%
    \@xp\let\@xp\subjclassname\csname subjclassname@#1\endcsname
  }%
}

\usepackage{graphicx,amsmath,amssymb}
\usepackage{algorithm}
\usepackage{mathdots, comment}
\usepackage[noend]{algpseudocode}

\newtheorem{theorem}{Theorem}[section]

\newtheorem*{theorem*}{Theorem}
\newtheorem{proposition}[theorem]{Proposition}

\newtheorem{corollary}[theorem]{Corollary}

\theoremstyle{definition}

\theoremstyle{remark}

\newtheorem{example}[theorem]{Example}

\begin{document}
\title[Structure and Decomposition of Deltoids]{Structure and Decomposition of Deltoids in Abelian Groups}

\author[M. Aliabadi, J. Losonczy]{Mohsen Aliabadi$^{1}$ \and Jozsef Losonczy$^{2,*}$}
\thanks{$^1$Department of Mathematics, Clayton State University, 
2000 Clayton State Boulevard, Lake City, Georgia 30260, USA.  \url{maliabadi@clayton.edu}.\\
$^2$Department of Mathematics, Long Island University,
720 Northern Blvd, Brookville, New York 11548, USA. \url{Jozsef.Losonczy@liu.edu}.}
\thanks{$^*$Corresponding Author.}


\begin{abstract}
Deltoids provide a natural framework for studying defective (partial) matchings in abelian groups, and we develop both structure and existence results in this setting. Given finite subsets \( A \) and \( B \) of an abelian group \( G \), a matching is a bijection \(f:A\to B\) such that \(af(a)\notin A\) for all \(a\in A\), a definition motivated by the study of canonical forms for symmetric tensors. We provide necessary and sufficient conditions for the existence of a partial matching with any prescribed defect, and then describe the minimal unavoidable defect for a pair \((A,B)\). We also define and examine a defective version of Chowla sets in the matching context.  We prove a structure theorem identifying obstructions to the existence of partial matchings with small defect. Finally, within the deltoid setup, we establish max--min results on the partitioning of \( A \) and \( B \) into left- and right-admissible sets.  Our tools mix results from transversal theory with ideas from additive number theory.
\end{abstract}
\maketitle

\section{Introduction} \label{Intro}

The study of matchings in groups originated in \cite{Losonczy 1}, motivated by a conjecture of Wakeford on canonical forms for symmetric tensors \cite{Wakeford}. Progress toward a resolution of the conjecture was made by establishing the existence of a specific kind of matching, called an acyclic matching, for certain pairs of subsets of $\mathbb{Z}^n$. This existence property was subsequently shown to hold in a more general sense for $\mathbb{Z}^n$~\cite{Alon}, then extended to torsion-free abelian groups \cite{Losonczy 2}, and ultimately characterized for all abelian groups \cite{Taylor}.  Further developments, encompassing a broader class of matchings, include strengthened results in abelian groups \cite{Aliabadi 0}, an extension to arbitrary (not necessarily abelian) groups \cite{Eliahou 1}, and a lower bound for the number of matchings \cite{Hamidoune}.  The theory has also been generalized to the setting of skew field extensions \cite{Eliahou 2} and to that of matroids \cite{Aliabadi 3}. 

Let \(G\) be an abelian group, with operation written multiplicatively. For nonempty finite sets \(A,B\subseteq G\), a \emph{matching} is a bijection \(f:A\to B\) such that \(af(a)\notin A\) for all \(a\in A\). The present authors have recently characterized matchable pairs \( (A,B) \) in abelian groups using Dyson's $e$-transform \cite{Aliabadi 5} and have obtained structure and existence theorems for pairs that are not matchable \cite{Aliabadi 6}. These results have helped to clarify the role played by periodic sets in the theory.  We continue this line of work here by investigating partially matchable pairs in abelian groups.

The paper is organized as follows. Section~\ref{Prelims} introduces a deltoid framework for partial matchings and recalls some relevant results from transversal theory. In Section~\ref{PMG}, we investigate how the existence of partial matchings of \((A,B)\) relates to the arithmetic characteristics of \(A\) and \(B\) (Theorem~\ref{Char for partial in groups} and Corollary~\ref{arithmetic with defect}).  Our analysis again leverages Dyson’s \(e\)-transform, which provides a link between combinatorial matching conditions and these arithmetic features. In Section~\ref{Structure section for partial}, a decomposition theorem (Theorem~\ref{Structure of PMG}) and a result revealing an obstruction to the existence of pairs \( (A,B) \) of prescribed size and deficiency (Theorem~\ref{Existence partial}) are presented. We study the problem of partitioning \(A\) and \(B\) into left- and right-admissible subsets in Section~\ref{Decompose section groups}.  Concluding remarks and directions for future work are offered in the final section.

\section{Preliminaries} \label{Prelims}

Let \( A \) and \( B \) be nonempty finite sets, and let \( \Delta \) be a subset of \( A \times B \). Following \cite{Mirsky}, we call the triple \( (A, \Delta, B) \) a {\em deltoid}. Note that \( A, B \) are not assumed to be disjoint.  We say that a subset \( A' \) of \( A \) is {\em left-admissible} if there exists an injective mapping \( f: A' \longrightarrow B \) such that \( (a, f(a)) \in \Delta \) for all \( a \in A' \).  We then also say that the set \( B' = f(A') \) is {\em right-admissible}, and that \( f \) is a {\em partial matching} of \( (A,\Delta,B) \) with {\em defect}  \( d = |A \setminus A'| \).  

We remark that the empty set is both left- and right-admissible, and the empty mapping (from \( \emptyset \) to \( B \)) is a partial matching of \( (A, \Delta, B) \).  A surjective partial matching of \( (A,\Delta,B) \) with defect \( 0 \), if it exists, is usually referred to simply as a {\em matching}. 

For \( S \subseteq A \), we define a subset \( \Delta(S) \) of \( B \) by
\[ 
\Delta(S) = \{ b \in B :  (a,b) \in \Delta \mbox{ for some }a \in S \}. 
\]

In this paper, we will use the following three well-known results from matching theory. The reader is referred to the monograph \cite{Mirsky} for proofs.\footnote{To any deltoid \( (A, \Delta, B) \), one can associate a family \( \mathcal{A} = (S_a : a \in A ) \) of subsets of \( B \) by setting \( S_a = \Delta(\{ a \}) \). A subset \( B' \) of \( B \) is called a {\it partial transversal} of \( \mathcal{A} \) if there is an injective mapping \( g: B' \longrightarrow A \) such that \( b \in S_{g(b)} \) for all \( b \in B' \). We then say that \( B' \) has {\it defect} \( d = |A\setminus g(B')| \). Theorems~\ref{Ore}, \ref{Partition A}, and \ref{Partition B} above can be obtained by applying Theorems~3.2.1, 3.3.2, and 3.3.5 in \cite{Mirsky} to the family \( \mathcal{A} \) and then translating back to the deltoid setting.}  It is assumed in these theorems that the sets \( A \) and \( B \) are nonempty and finite.  The first result is sometimes called the ``defect form" of P.\ Hall's marriage theorem.

\begin{theorem}[Hall--Ore]  \label{Ore}
Let \( (A,\Delta,B) \) be a deltoid and let \( 0 \leq d \leq |A| \) be an integer. Then there exists a partial matching of \( (A,\Delta,B) \) with defect \( d \) if and only if, for every subset \( S \) of \( A \), we have 
\[ 
|S| - d \leq | \Delta(S)|. 
\]
\end{theorem}

The next theorem gives a simple necessary and sufficient condition for the existence of a decomposition of \( A \) into pairwise disjoint left-admissible sets. We will use this result in Section~\ref{Decompose section groups}.

\begin{theorem}  \label{Partition A}
Let \( (A,\Delta,B) \) be a deltoid. Let \( k \) be a positive integer. Then \( A \) can be written as a union of \( k \) pairwise disjoint left-admissible sets if and only if, for every subset \( S \) of \( A \), we have 
\[ 
|S| \leq  k| \Delta(S)|. 
\]
\end{theorem}

The criterion for the decomposition of \( B \) stated below has a slightly more complicated form, but it will prove to be easy to work with for the deltoids considered in this paper.

\begin{theorem}  \label{Partition B}
Let \( (A,\Delta,B) \) be a deltoid.  Let \( k \) be a positive integer. Then \( B \) can be written as a union of \( k \) pairwise disjoint right-admissible sets if and only if, for every subset \( S \) of \( A \), we have 
\[ 
k|S| + |B| \leq  k|A| + |\Delta(S)|. 
\]
\end{theorem}

We are interested in studying the partial matchings and admissible sets of certain algebraically defined deltoids.
Let \( G \) be an abelian group (its operation written multiplicatively).  Given subsets \( S, T \subseteq G \), we define \( ST \) by
\[
ST = \{ u \in G : u = st \mbox{ for some }s \in S\mbox{ and }t \in T \}.
\]
Clearly, if at least one of the sets \( S, T \) is empty, then so is \( ST \).  In the case where \( S \) is a singleton, say \( S = \{ x \} \), we will also write \( xT \) for \( ST \), and likewise if \( T \) is a singleton.

Let \( A \) and \( B \) be nonempty finite subsets of \( G \), with \( |A| = |B| \).  
We form a deltoid \( (A, \Delta, B) \) by defining 
\[  
\Delta = \{ (a,b) \in A \times B : ab \notin A \}. 
\]
Henceforth, our deltoids will be as above. Thus a partial matching of \( (A,\Delta,B) \) with defect \( d \) is an injective mapping \( f : A' \longrightarrow B \), where \( A' \subseteq A \) and \( d = |A \setminus A'| \), such that \( af(a) \notin A \) for all \( a \in A' \). We will often omit mention of \( \Delta \) and refer to such a mapping as a ``partial matching of \( (A,B) \) with defect \( d \)." When there exists a partial matching with defect 0, i.e., a matching, we say that the pair \( (A,B) \) is {\em matchable}.  Otherwise, we call \( (A,B) \) {\em unmatchable}.  

Note that, with the above (permanent) assumption about \( \Delta \), we have, for each \( S \subseteq A\), 
\[ \Delta(S) = \{ b \in B : Sb \not\subseteq A \}. \]

\section{Characterization of partially matched pairs}\label{PMG}

We wish to explore the connections between the partial matchings of \( (A,B) \) and the arithmetic structure of \( A \) and \( B \). Unfortunately, Theorem~\ref{Ore} is not, on its own, a convenient tool for this purpose.  Previous work in this area has relied heavily on supplementary ideas, mostly in the form of certain inequalities from additive number theory. To get around this, we will derive a more specialized version of Theorem~\ref {Ore}.  This will be accomplished through the use of Dyson's \( e \)-transform.  The proposition below (or rather, its proof) explains exactly how the \( e \)-transform will be used.  A general discussion of this tool is given in \cite{Nathanson}.

\begin{proposition} \label{e-transform for groups}
Let \( A \), \( S \), and  \( R \) be nonempty finite subsets of an abelian group \( G \). Assume that \( 1 \in R \) and \( SR \subseteq A \). Then there exist sets \( S' \) and \( R' \) such that the following three conditions hold:
\begin{itemize}
\item[(i)]  \( S \subseteq S'R' = S' \subseteq A \),
\item[(ii)]  \( 1 \in R' \subseteq R \),
\item[(iii)] \( |S'| + |R'| = |S| + |R| \).
\end{itemize}
\end{proposition}

\begin{proof}
First observe that if \( SR = S \), we can simply take \( S' = S \) and \( R' = R \).  Assume that \( SR \neq S \). Since  \( 1 \in R \), we have \( SR \not\subseteq S \), and so there exist \( e \in S \) and \( r \in R \) such that \( er \notin S \).  We define sets \( S_1 \) and \( R_1 \) as follows:
\begin{align*}
 S_1 &= S \cup (eR), \\
R_1 &= R \cap (Se^{-1}).
\end{align*}
The sets \( S_1 \) and \( R_1 \) are known as the \( e \)-transforms of \( S \) and \( R \).

We claim that the four conditions below hold:
\begin{itemize}
\item[(a)] \( S_1R_1 \subseteq SR \subseteq A \),
\item[(b)] \( 1 \in R_1 \subseteq R \),
\item[(c)] \( |S_1| + |R_1| = |S| + |R| \),
\item[(d)] \( S \subseteq S_1 \subseteq A \mbox{ and }|S| < |S_1|. \)
\end{itemize}
Conditions (a) and (b), as well as the inclusion \( S \subseteq S_1 \) in (d), are immediate. The other inclusion in (d) follows from (a) and (b) (specifically, \( S_1R_1 \subseteq A \) and \( 1 \in R_1 \)).  The rest of (d) is a consequence of the fact that \( S \subseteq S_1 \) and \( er \in S_1 \setminus S \).  To see that (c) holds, we first compute 
\begin{align*}
|S_1| &= |S \cup (eR)| \\
&= |S| + |eR| - |S \cap (eR)| \\
&= |S| + |R| - |S \cap (eR)|,
\end{align*}
and then observe that the mapping 
\[ R_1 \longrightarrow S \cap (eR) \]
\[ x \mapsto ex \]
is a bijection. 

If \( S_1R_1 \neq S_1 \), the above is repeated, this time with  \( S_1 \) and \( R_1 \) replacing \( S \) and \( R \), respectively.  We continue the procedure until we reach nonempty sets \( S_n \) and \( R_n \) satisfying \( S_nR_n = S_n \). This is guaranteed to occur because the sequence \( |S_1|, |S_2|, \ldots \) is strictly increasing and the set \( A \) is finite. Let \( S' = S_n \) and \( R' = R_n \). It is clear that these sets satisfy the conditions in the statement. 
\end{proof}

The following result establishes the equivalence of a pair of conditions involving real parameters \( \alpha \) and \( \beta \).  When these parameters are assigned certain values, the conditions become criteria for the existence of partial matchings.  For certain other values, we obtain tests for the existence of a partition of \( B \) into admissible sets.

\begin{theorem}  \label{Equivalence for groups}
Let \( A \) and \( B \) be nonempty finite subsets of an abelian group \( G \), with \( |A| = |B| \) and \( 1 \notin B \). Let \( \alpha \geq 1 \) and \( \beta \) be real numbers. Then the following two conditions are equivalent:
\begin{itemize}
\item[(i)] For every subset \( S \) of \( A \), we have \( \alpha |S| + \beta \leq | \Delta(S)| \), where \( \Delta(S) = \{ b \in B : Sb \not\subseteq A \} \).
\item[(ii)] For every pair of subsets \( S \subseteq A \) and \( R \subseteq B \cup \{ 1 \} \) with \( SR = S \), we have \( \alpha |S| + \beta \leq |B \setminus R| \).
\end{itemize}
\end{theorem}

\begin{proof}
Assume that (i) holds.  Let \( S \) and \( R \) be sets such that \( S \subseteq A \), \( R \subseteq B \cup \{ 1 \} \), and \( SR = S \).  Let \( U_S = \{ b \in B : Sb \subseteq A \} \) and note that \( \Delta(S) = B \setminus U_S \).  We also have \( R \subseteq U_S \cup \{ 1 \} \) (for if \( x \in R \) and \( x \neq 1 \), then \( x \in B \) and \( Sx \subseteq SR = S \subseteq A \), so \( x \in U_S \)). Hence \( B \setminus (U_S \cup \{ 1 \}) \subseteq B \setminus R \).  Since \( 1 \notin B \), this last inclusion can be rewritten as \( B \setminus U_S \subseteq B \setminus R \), which gives us \( |B \setminus U_S| \leq |B \setminus R| \).  Combining this with the assumption that \( \alpha |S| + \beta \leq |\Delta(S)| = |B \setminus U_S| \), we obtain \( \alpha |S| + \beta \leq |B \setminus R| \), as desired.

Conversely, assume that (ii) holds. Let \( S \) be a subset of \( A \), and let \( U_S \) be defined as above. Put \( R = U_S \cup \{ 1 \} \).  Our argument splits into two cases.

\medskip

\textbf{Case 1:} \( SR = S \). The hypothesis can be applied to \( S \) and \( R \), yielding \( \alpha |S| + \beta \leq |B \setminus R| \).  Since \( 1 \notin B \), we have \( B \setminus R = B \setminus U_S \), and therefore 
\( \alpha |S| + \beta \leq |B \setminus U_S| = |\Delta(S)| \). 

\medskip

{\bf Case 2:} \( SR \neq S\). Note that, under this assumption, \( S \) must be nonempty. We use Dyson's \( e \)-transform, via Proposition~\ref{e-transform for groups}, to obtain sets \( S' \) and \( R' \) such that 
\begin{itemize}
\item[(a)] \( S \subseteq S'R' = S' \subseteq A \), 
\item[(b)] \( 1 \in R' \subseteq R \subseteq B \cup \{ 1 \} \), 
\item[(c)] \( |S'| + |R'| = |S| + |R| \).
\end{itemize}
Observe that our hypothesis can be applied to the sets \( S' \) and \( R' \). This gives us \( \alpha |S'| + \beta \leq |B \setminus R'| \).  Again using the fact that \( 1 \notin B \), we rewrite this inequality as 
\[ 
\alpha |S'| + \beta \leq |B \setminus (R'\setminus \{ 1 \})|.
\]
By (b), this implies 
\[ 
\alpha |S'| + \beta \leq |B| - |R'| + 1,
\]
which we can rewrite as 
\begin{equation} 
\alpha |S'| + |R'| + \beta \leq |B| + 1. \label{Case2 Ineq}
\end{equation}
Now, from conditions (b) and (c), we know that 
\[
0 \leq |R| - |R'| = |S'| - |S|.
\]
Since \( \alpha \geq 1 \), this can be extended to
\[ 
0 \leq |R| - |R'| \leq |S'| - |S| \leq \alpha (|S'| - |S|), 
\]
which implies
\[
\alpha |S| + |R| \leq \alpha |S'| + |R'|.
\]
Adding \( \beta \) to each side and then combining the result with (\ref{Case2 Ineq}), we get
\[
\alpha |S| + |R| + \beta \leq |B| + 1,
\]
or
\[
\alpha |S| + \beta \leq |B| - |R| + 1.
\]
Finally, since \( U_S \subseteq B \) and  \( |U_S| = |R| -1 \), it follows that
\[
\alpha |S| + \beta \leq |B| - |U_S| = |B \setminus U_S| = | \Delta(S) |, 
\]
as desired.
\end{proof}

We are ready to present the main result of this section, Theorem~\ref{Char for partial in groups}. Several consequences will follow, including a decomposition theorem for unmatchable pairs in Section~\ref{Structure section for partial}.

\begin{theorem}  \label{Char for partial in groups}
Let \( A \) and \( B \) be nonempty finite subsets of an abelian group \( G \), with \( |A| = |B| \) and \( 1 \notin B \). Let \( 0 \leq d \leq |A| \) be an integer. Then there exists a partial matching of \( (A,B) \) with defect \( d \) if and only if, for every pair of subsets \( S \subseteq A \) and \( R \subseteq B \cup \{ 1 \} \) with \( SR = S \), we have 
\[ |S| - d \leq |B \setminus R|. 
\]
\end{theorem}

\begin{proof}
We simply apply Theorems~\ref{Ore} and \ref{Equivalence for groups}, taking \( \alpha = 1 \) and \( \beta = -d \) in the latter theorem.
\end{proof}

If \( R \) is a subset of a group \( G \), we write \( \langle R \rangle \) for the subgroup of \( G \) generated by \( R \). The proposition below explains exactly how the condition \( SR = S \) will be used in this paper. A proof of this well-known algebraic fact appears in \cite{Aliabadi 5}.

\begin{proposition} \label{UnionOfCosets}
Let \( G \) be an abelian group and let \( S \) and \( R \) be nonempty finite subsets of \( G \). Then \( SR = S \) if and only if \( S \) is a union of cosets of \( \langle R \rangle \). 
\end{proposition}

Consider nonempty finite sets $A$ and $B$ with $|A|=|B|$ and $1\notin B$ in an abelian group \( G \), and define \( \delta(A,B) \) to be the smallest nonnegative integer \( d \) such that \( (A,B) \) has a partial matching with defect \( d \).  We call \( \delta(A,B) \) the {\it deficiency} of the pair \( (A,B) \). Note that \( \delta(A,B) = |A| \) if the only partial matching is the empty one.  It is an immediate consequence of the Ore deficiency theorem (see Theorem~1.3.1 in \cite{Lovasz}) that \( \delta(A,B) \) equals the maximum value of the difference
\[
|S| - |\Delta(S)|,
\]
where \( S \) ranges over all subsets of \( A \). For our first application of Theorem~\ref{Char for partial in groups}, we present an alternate formula for \( \delta(A,B) \).

\begin{corollary}
Let \( A \) and \( B \) be nonempty finite subsets of an abelian group \( G \), with \( |A|=|B| \) and \( 1 \notin B \). Then \( \delta(A,B) \) equals the maximum value of 
\[
|S|-|B \setminus R|,
\]
where \( S \) and \( R \) range over all subsets of \( A \) and \( B \cup\{1\} \), respectively, such that \( SR = S \).
\end{corollary}

\begin{proof}
Let $d=\delta(A,B)$ and define \( \mathcal{I} \) to be the set of all differences \( |S|-|B \setminus R| \), where \( S \) and \( R \) are as in the statement.  By Theorem~\ref{Char for partial in groups}, since $(A,B)$ admits a partial matching with defect $d$, it follows that for all such \( S \) and \( R \), 
\[
|S|-d\ \le\ |B\setminus R|,
\]
hence $d\ \ge\ |S|-|B\setminus R|$. In particular, $d \ge \max \mathcal{I}$.

For the reverse inequality, let $d'= \max \mathcal{I}$. Then for all $S\subseteq A$ and $R\subseteq B\cup\{1\}$ with $SR=S$, we have
\[
|S|-d'\ \le\ |B\setminus R|.
\]
Applying Theorem~\ref{Char for partial in groups} again, we obtain that $(A,B)$ admits a partial matching with defect $d'$. By the minimality of $d=\delta(A,B)$, we must have $d\le d'$. Combining with the first inequality gives $d=d'$, as claimed.
\end{proof}

In the next corollary, we present a generalization of the fact, first established in \cite{Losonczy 2}, that \( (A,A) \) is matchable provided \( 1 \notin A \).

\begin{corollary} \label{Intersection in Groups}
Let \( A \) and \( B \) be nonempty finite subsets of an abelian group \( G \), with \( |A| = |B| \) and \( 1 \notin A \cup B \). Let \( d = |A| - |A \cap B| \). Then there exists a partial matching of \( (A,B) \) with defect \( d \). 
\end{corollary}

\begin{proof}
Let \( S \) and \( R \) be sets satisfying \( S \subseteq A \), \( R \subseteq B \cup \{ 1 \} \), and \( SR = S \). We will show that \( |S| - d \leq |B \setminus R|\).

\medskip

We claim that \( S \cap R = \emptyset \).  To see this, assume the contrary and let \( a \in  S \cap R \).  Then \( a \langle R \rangle = \langle R \rangle \), and \( a \langle R \rangle \subseteq S \) by Proposition~\ref{UnionOfCosets}. Therefore \( 1 \in S \subseteq A \), a contradiction.

\medskip

Using the claim and the fact that \( S \cup (R \setminus \{ 1 \} ) \subseteq A \cup B \), we find that 
\[
|S| + | R \setminus \{ 1 \} | \leq |A \cup B| = |A| + |B| - |A \cap B| = |B| + d,
\]
hence
\[
|S| - d \leq |B| - | R \setminus \{ 1 \} | = |B \setminus (R \setminus \{ 1 \}) | = |B \setminus R|,
\]
since \( 1 \notin B \). Applying Theorem~\ref{Char for partial in groups} completes the proof.
\end{proof}

Let \( G \) be an abelian group. We denote the order of any \( x \in G \) by \( o(x) \). A {\em progression} of {\em length} \( n \) in \( G \) is a sequence of the form \( a, ax, \ldots , ax^{n-1} \) where \( a, x \in G \) and \( 0 \leq n - 1 < o(x) \). If \( n > 1 \), we call \( x \) the {\em ratio} of the progression.  The result below reveals a connection between the partial matchability of \( (A,B) \) and the arithmetic characteristics of \( A \) and \( B \).

\begin{corollary} \label{arithmetic with defect}
Let \( A \) and \( B \) be nonempty finite subsets of an abelian group \( G \), with \( |A| = |B| \) and \( 1 \notin B \).  Let \( d \) and \( n \) be integers such that \( 0 \leq d \leq |B| \) and \( n > 0 \).  Assume that \( A \) contains no progression with ratio in \( B \) and of length greater than \( n \). Also, assume that all but at most \( d \) elements of \( B \) have order greater than \( n \).  Then there exists a partial matching of \( (A,B) \) with defect \( d \).
\end{corollary}

\begin{proof}
Let \( S \) and \( R \) be sets such that \( S \subseteq A \), \( R \subseteq B \cup \{ 1 \} \), and \( SR = S \).  As in the previous proof, we will show that \( |S| - d \leq |B \setminus R| \). We may assume that \( S \) is nonempty.

\medskip

We claim that \( | R \setminus \{ 1 \} | \leq d \).  To see this, suppose \( x \in R \setminus \{ 1 \} \), and let \( a \in S \).  By Proposition~\ref{UnionOfCosets}, we have \( a \langle x \rangle \subseteq a \langle R \rangle \subseteq S \).  It follows that \( o(x) \leq n \), since otherwise the sequence \(  ax, ax^2, \ldots ,ax^{n+1} \) would be a progression in \( S \subseteq A \) with ratio in \( B \) and of length greater than \( n \).  By hypothesis, all but at most \( d \) elements of \( B \) have order greater than \( n \), so the claim is proved.

\medskip

It follows that, since \( 1 \notin B \), 
\[
|B \setminus R| = |B \setminus (R \setminus \{ 1 \} )| = |B| - | R \setminus \{ 1 \} | \geq |B| - d = |A| - d \geq |S| - d.
\]
We conclude by applying Theorem~\ref{Char for partial in groups}.
\end{proof}

Let \( G \) be a group and let \( d \) be a nonnegative integer.  A nonempty finite subset \( B \) of \( G \) is called a \emph{$d$-defective Chowla set} if there exists \( B_0 \subseteq B \) with \( |B_0| \geq |B| - d \) such that \( o(x) > |B| \) for every \( x \in B_0 \).  A $0$-defective Chowla set is better known as a Chowla set \cite{Hamidoune}. The corollary below generalizes Theorem~1.1 in \cite{Hamidoune} in the case where \( G \) is abelian.

\begin{corollary}\label{DefChowla}
Let \( A \) and \( B \) be nonempty finite subsets of an abelian group \( G \), with \( |A| = |B| \) and \( 1 \notin B \).  Let \( 0 \leq d \leq |B| \) be an integer.  Assume that \( B \) is a $d$-defective Chowla set. Then there exists a partial matching of \( (A,B) \) with defect \( d \).
\end{corollary}

\begin{proof}
By the definition of $d$-defective Chowla set, all but at most \( d \) elements of \( B \) have order greater than \( |B| \).  We now simply apply Corollary~\ref{arithmetic with defect}, taking \( n = |A| = |B| \).
\end{proof}

\section{Structure and existence theorems} \label{Structure section for partial}

Below, we establish a decomposition theorem for pairs \( (A,B) \) having positive deficiency.  This result can be viewed as a refinement of Theorem~2.3 in \cite{Aliabadi 6}.  It will be used in the proof of Theorem~\ref{Existence partial} and in our study of admissible sets in Section~\ref{Decompose section groups}.

\begin{theorem} \label{Structure of PMG}
Let \( A \) and \( B \) be nonempty finite subsets of an abelian group \( G \), with \( |A| = |B| \) and \( 1 \notin B\).  Let \( \ell \) be a nonnegative integer.  Then \( \delta(A,B) > \ell \) if and only if there exists a nonempty subset \( R \) of \( B \) such that \( A \) and \( B \) can be expressed as disjoint unions as follows:
\[
A = S \cup Y, \qquad 
B = R \cup Z,
\]
where \( S \) is a nonempty union of cosets of \( \langle R \rangle \), and \( Y \)satisfies \( |Y| < | R | - \ell \). 
\end{theorem}

\begin{proof}
Our argument is a modified version of the proof of Theorem~2.3 in \cite{Aliabadi 6}.  Assume that \( \delta(A,B) > \ell \). Then there is no partial matching of \( (A,B) \) with defect \( \ell \), and so, by Theorem~\ref{Char for partial in groups}, there exist sets \( S \) and \( R \) satisfying the following conditions: 
\begin{itemize}
\item[(a)] \( S \subseteq A \) and  \( R \subseteq B \cup \{ 1 \} \),
\item[(b)] \( SR = S \), 
\item[(c)] \( |S| > |B \setminus R | + \ell \).
\end{itemize}
By (c), we have \( S \neq \emptyset \), and then (b) gives us \( R \neq \emptyset \) as well.  From condition (b) and Proposition~\ref{UnionOfCosets}, we see that \( S \) is a union of cosets of \( \langle R \rangle \).  Now, \( R \neq \{ 1 \} \) by (c) again. Moreover, we may assume that  \( 1 \notin R \), since otherwise we can replace \( R \) with \( R \setminus \{ 1 \} \), and (a)--(c) will still hold.

Thus \( R \subseteq B \), and so condition (c) becomes
\[
|S| > |B| - |R| + \ell.
\]
Take \( Y = A \setminus S \) and \( Z = B \setminus R \). Since \( |A| = |B| \) and \( S \subseteq A \), the above inequality shows that \( |Y| < |R| - \ell \), giving us the desired decomposition.

Conversely, assume that \( A \) and \( B \) can be written as in the statement. Then the sets \( S \) and \( R \) satisfy \( S \subseteq A \), \( R \subseteq B \), and \( SR = S \).  Furthermore, the inequality \( |Y| < | R | - \ell \) gives us 
\[ |S| + |R| - \ell > |S| + |Y| = |A| = |B|, 
\]
hence \( |S| - \ell > |B \setminus R|\). Applying Theorem~\ref{Char for partial in groups} (with \( d = \ell\)), we find that there is no partial matching of \( (A,B) \) with defect \( \ell \), and so \( \delta(A,B) > \ell \).
\end{proof}

\begin{example} \label{Structure Example}
In the additive group \( G = \mathbb{Z}/12\mathbb{Z} \), consider the sets
\[ 
A = \{ \overline{0}, \overline{1}, \overline{2}, \overline{4}, \overline{6}, \overline{8}, \overline{10}, \overline{11} \},
\qquad
B = \{ \overline{1}, \overline{2}, \overline{3}, \overline{4}, \overline{6}, \overline{8}, \overline{10}, \overline{11} \}.
\]
For the subset \( R = \{ \overline{2}, \overline{4}, \overline{6}, \overline{8}, \overline{10} \} \) of \( B \), we can write \( A \) and \( B \) as disjoint unions
\[
A = S \cup Y, \qquad B = R \cup Z,
\]
where \( S = \langle R \rangle \), \( Y = \{ \overline{1}, \overline{11} \} \), and \( Z = B \setminus R \). Observe that \( |Y| < |R| - 2 \), and so \( \delta(A,B) > 2 \) by Theorem~\ref{Structure of PMG}.  In fact, \( \delta(A,B) = 3 \), since the mapping \( A \setminus \{\overline{4}, \overline{6}, \overline{8} \} \longrightarrow B \) given by \( \overline{0}\mapsto \overline{3} \), \( \overline{1}\mapsto \overline{2} \), \( \overline{2}\mapsto \overline{1} \), \( \overline{10}\mapsto \overline{11} \), \( \overline{11}\mapsto \overline{10} \) is a partial matching of \( (A,B) \) with defect $3$.
\end{example}

For any group \( G \) with at least one finite nontrivial proper subgroup, we let \( n_0(G) \) denote the smallest size of such a subgroup.  In the theorem below, we provide a necessary and sufficient condition for the existence of subsets \( A,B \subseteq G \) of prescribed size with \( 1 \notin B \) and deficiency \( \delta(A,B) \) greater than a given value $\ell$.

\begin{theorem} \label{Existence partial}
Let $G$ be an abelian group with at least one finite nontrivial proper subgroup. Let \( \ell \) and \( n \) be integers such that \( \ell \ge 0 \) and \( n_0(G) \le n<|G| \). Then there exist subsets $A,B\subseteq G$ such that \( 1\notin B \), \( |A|=|B|=n \), and \( \delta(A,B) > \ell \) if and only if there is a subgroup \( H \) of \( G \) with \( |H|\le n \) and \( |H|\nmid (n+j) \) for all \( j = 1, \ldots, \ell+1 \).
\end{theorem}

\begin{proof}
Assume there exist sets \( A \) and \( B \) as in the statement.  By Theorem \ref{Structure of PMG}, we can write these as disjoint unions
\[
A = S \cup Y, 
\qquad 
B = R \cup Z,
\]
where $R$ is nonempty, $S$ is a nonempty union of cosets of $\langle R\rangle$, and $|Y| < |R|-\ell$.  Put $H=\langle R\rangle$, and note that \( H \) is finite as \( S \) is nonempty and finite. In fact, \( |H| \leq |S| \leq n \).  Let $m=|H|$, and write $|S|=mq$ where $q$ is a positive integer. We have $R\subseteq H\setminus\{1\}$, so $|R|\le m-1$. Therefore
$$
|Y|<|R|-\ell \le (m-1)-\ell,
$$ 
which gives 
\[
|Y|\le m-\ell-2.
\]
We also have
\[
n=|A|=|S|+|Y|=mq+|Y|.
\]
Hence, for each $t\in\{1,\dots,\ell+1\}$, 
\[
n+t = mq + (|Y|+t)
\]
and
\[
|Y|+t \le (m-\ell-2)+(\ell+1)=m-1,
\]
so the remainder of $n+t$ modulo $m$ is never $0$. This implies
\[
m\nmid (n+t)\qquad \text{for } t=1,\dots,\ell+1,
\]
which shows that $H$ is the desired subgroup.

Conversely, assume there exists a subgroup \( H \) of  \( G \) with \( |H|=m \le n \) and $m\nmid(n+j)$ for all \( j = 1, \ldots, \ell+1 \).
Write $n=mq+r$ with integers $q\ge 1$ and $0\le r<m$. We claim that
\[
r <  m-\ell-1.
\]
Assume instead that \( r + \ell + 1 \geq m \). Then the interval $\{r+1,\dots,r+\ell+1\}$ contains $m$, say $m=r+j_0$ where $1\le j_0\le \ell+1$. Since $n+j_0 \equiv r+j_0 \pmod m$, this gives $n+j_0\equiv 0\pmod m$, i.e., $m\mid(n+j_0)$, contradicting the assumption. The claim is proved.

We now construct $A$ and $B$ in the spirit of Theorem~2.9 of~\cite{Aliabadi 6}. Let \( R = H\setminus \{1\} \) (so $|R|=m-1$), and let $S$ be the union of $q$ distinct cosets of $H$ (so $|S|=mq$).
Choose a set $Y\subseteq G\setminus S$ with $|Y|=r$, and choose a set
$Z\subseteq G\setminus (R \cup \{1\})$ with
\[
|Z|=n-(m-1)=n-m+1.
\]
Set $A=S \cup Y$ and $B=R \cup Z$.
Then $|A|=|B|=n$ and $1\notin B$. Moreover,
\[
|Y|=r < m-\ell -1 = |R|-\ell.
\]
Since $S$ is a union of cosets of
\[
\langle R\rangle = \langle H \setminus \{1\} \rangle = H,
\]
Theorem~\ref{Structure of PMG} applies and yields $\delta(A,B) > \ell$, completing the proof.
\end{proof}

\section{Decomposition into admissible subsets}  \label{Decompose section groups}

Let \( A \) and \( B \) be nonempty finite subsets of an abelian group, with \( |A| = |B| \) and \( 1 \notin B \). In this section, we investigate the possibility of partitioning \( A \) and \( B \) into admissible subsets.  Recall, a subset of \( A \) is left-admissible if it is the domain of some partial matching of \( (A,B) \), and the range of such a partial matching is called a right-admissible subset of \( B \). 

We begin by defining the {\em right partition number} of \( (A,B) \), denoted by \( \rho = \rho(A,B) \).  If \( Ax \neq A \) for all \( x \in B \), we let \( \rho \) be the maximum value of the expression
\[
\left\lceil \frac{|U_S|}{|A| - |S|} \right\rceil,
\]
where \( S \) ranges over all proper subsets of \( A \), and, as usual, \( U_S = \{ b \in B : Sb \subseteq A \}. \) If instead \( Ax = A \) for some \( x \in B \), we set \( \rho = \infty \).

Our first result concerning the right partition number is given below.  Recall that for any \( S \subseteq A \), we have \( \Delta(S) = B \setminus U_S \).

\begin{theorem}  \label{Decomposition of B}
Let \( A \) and \( B \) be nonempty finite subsets of an abelian group \( G \), with \( |A| = |B| \) and \( 1 \notin B \).  Let \( \rho \) be the right partition number of \( (A,B) \).  Then \( B \) can be written as a union of right-admissible sets if and only if \( \rho < \infty \).  Moreover, when \( \rho < \infty \), \( B \) can be written as a union of \( \rho \), and no fewer, pairwise disjoint right-admissible sets.
\end{theorem}

\begin{proof}
Suppose \( \rho = \infty \). Then \( Ax = A \) for some \( x \in B \). Such an \( x \) cannot belong to a right-admissible subset of \( B \), hence \( B \) cannot be expressed as a union of right-admissible sets.

Suppose \( \rho < \infty \).  Then \( \Delta(A) = B \).  We will show that the condition in Theorem~\ref{Partition B} holds for \( k = \rho \).  Let \( S \) be a subset of \( A \). If \( S = A \), then the inequality in Theorem~\ref{Partition B} clearly holds (we have equality in this instance):
\[ 
\rho|A| + |B| = \rho|A| + |\Delta(A)|.
\]
Suppose \( S \neq A \).  Then
\[
\frac{|U_S|}{|A| - |S|} \leq \rho,
\]
hence
\[
\rho |S| + |U_S| \leq \rho |A|.
\]
Adding \( |B| - |U_S| \) to both sides, we get
\[
\rho |S| + |B|  \leq \rho |A| + |B| - |U_S|.
\]
Observe that \( |B| - |U_S| = |B \setminus U_S| = |\Delta(S)| \).  Therefore, by Theorem~\ref{Partition B}, \( B \) can be written as a union of \( \rho \) pairwise disjoint right-admissible sets.  

To establish the asserted minimality property of \( \rho \), suppose that \( B \) can be written as a pairwise disjoint union of \( k \) right-admissible sets. Then, by Theorem~\ref{Partition B}, for any proper subset \( S \) of \( A \), 
\[
k|S| + |B| \leq k|A| + | \Delta(S) |.
\]
Solving for \( k \) and using the fact that \( \Delta(S) = B \setminus U_S \), we find that 
\[
k \geq \frac{|U_S|}{|A| - |S|}.
\]
Applying the ceiling function to both sides yields
\[
k \geq \left\lceil \frac{|U_S|}{|A| - |S|}\right\rceil.
\]
It now follows from the definition of \( \rho \) that \( k \geq \rho \).
\end{proof}

Next, we present a specialized version of Theorem~\ref{Partition B}, tailored to the deltoids under consideration and involving the useful condition \( SR = S \).

\begin{theorem}  \label{Char for partition in groups}
Let \( A \) and \( B \) be nonempty finite subsets of an abelian group \( G \), with \( |A| = |B| \) and \( 1 \notin B \). Let \( k \) be a positive integer. Then \( B \) can be written as a union of \( k \) pairwise disjoint right-admissible sets if and only if, for every pair of subsets \( S \subseteq A \) and \( R \subseteq B \cup \{ 1 \} \) with \( SR = S \), we have 
\[ 
k|S| - (k-1)|A| \leq  |B \setminus R|. 
\]
\end{theorem}

\begin{proof}
We apply Theorems~\ref{Partition B} and \ref{Equivalence for groups}, this time taking \( \alpha = k \) and \( \beta = - (k-1)|A| \) in the latter and using the fact that \( |A| = |B| \).
\end{proof}

Using the above result, we obtain the following alternate formula for the right partition number of \( (A,B) \).

\begin{theorem}  \label{Formula for right partition}
Let \( A \) and \( B \) be nonempty finite subsets of an abelian group \( G \), with \( |A| = |B| \) and \( 1 \notin B \). Let \( \rho \) be the right partition number of \( (A,B) \).  Assume that \( \rho < \infty \).  Then \( \rho \) equals the maximum value of
\[
\left\lceil \frac{|R \setminus \! \{ 1 \}|}{|A| - |S|} \right\rceil,
\]
where \( S \) and \( R \) range over all subsets of \( A \) and \( B \cup \{ 1 \} \), respectively, such that \( SR = S \) and \( S \neq A \).
\end{theorem}

\begin{proof}
By Theorems~\ref{Decomposition of B} and \ref{Char for partition in groups}, \( \rho \) is the smallest positive integer \( k \) satisfying
\[ 
k|S| - (k-1)|A| \leq  | B \setminus R|
\]
for all pairs of subsets \( S \subseteq A \) and \( R \subseteq B \cup \{ 1 \} \) such that \( SR = S \).  A simple calculation shows that the above inequality holds if and only if 
\begin{equation} \label{middle}
k(|A| - |S|) \geq |R \setminus \{ 1 \}|,
\end{equation}
since \( |A| = |B| \) and \( |B \setminus R| = |B \setminus (R \setminus \{ 1 \})| = |B| - |R \setminus \{ 1 \}| \). Now, regarding the sets \( S \) and \( R \), we can have \( S = A \) only if \( R = \{ 1 \}  \). This is due to the equation \( SR = S \) and the assumption \( \rho < \infty \) (which means \( Ax \neq A \) for all \( x \in B \)). Inequality (\ref{middle}) clearly holds when \( R = \{ 1 \} \), so we now know that \( \rho \) is the smallest positive integer \( k \) satisfying (\ref{middle}) for all subsets \( S \subseteq A \) and \( R \subseteq B \cup \{ 1 \} \) such that \( SR = S \) and \( S \neq A \). Dividing by \( |A| - |S| \), we find that \( \rho \) is the smallest positive integer \( k \) satisfying
\[
k \geq \frac{|R \setminus \! \{ 1 \}|}{|A| - |S|}
\]
for all such \( S \) and \( R \).  The conclusion follows.
\end{proof}

Below, we show that the decomposition of \( A \) and \( B \) described in Theorem~\ref{Structure of PMG} provides information about the right partition number.

\begin{theorem}  \label{Estimate right partition}
Let \( A \) and \( B \) be nonempty finite subsets of an abelian group \( G \), with \( |A| = |B| \) and \( 1 \notin B \). Let \( \rho \) be the right partition number of \( (A,B) \).  Assume that the pair \( (A,B) \) is unmatchable, i.e., \( \delta(A,B) > 0\), and consider any decomposition \( A = S \cup Y \), \( B = R \cup Z \) as in Theorem~\ref{Structure of PMG}, with \( |Y| < |R| -  \ell\) for some nonnegative integer \( \ell \). Also, assume that \( \rho < \infty \). Then \( Y \) is nonempty and 
\[ 
\rho \geq \left\lceil \frac{|R|}{|Y|} \right\rceil > \frac{|R|}{|R| - \ell}.
\]
\end{theorem}

\begin{proof}
If \( Y \) were empty, then \( A \) would equal \( S \), which is a nonempty union of cosets of \( \langle R \rangle \). Hence we would have \(Ax = A \) for all \( x \in R \), implying \( \rho = \infty \), contrary to assumption. Thus \( |A| - |S| \) is nonzero. We also have \( SR = S \), by Proposition~\ref{UnionOfCosets}, and \( 1 \notin R \) since \( 1 \notin B \).   Therefore, by Theorem~\ref{Formula for right partition}, 
\[
\rho \geq \left\lceil \frac{|R|}{|A| - |S|} \right\rceil.
\]
We now simply observe that \( |A| - |S| = |Y| < |R| - \ell \). 
\end{proof}

\begin{example}
Consider again the setup in Example~\ref{Structure Example}, consisting of the sets
\[ 
A = \{ \overline{0}, \overline{1}, \overline{2}, \overline{4}, \overline{6}, \overline{8}, \overline{10}, \overline{11} \},
\qquad
B = \{ \overline{1}, \overline{2}, \overline{3}, \overline{4}, \overline{6}, \overline{8}, \overline{10}, \overline{11} \}
\]
in the additive group \( G = \mathbb{Z}/12\mathbb{Z} \).  As we saw, for the subset \( R = \{ \overline{2}, \overline{4}, \overline{6}, \overline{8}, \overline{10} \} \) of \( B \), we have the decomposition
\[
A = \langle R \rangle \cup Y, \qquad B = R \cup Z,
\]
where \( Y = \{ \overline{1}, \overline{11} \} \), \( Z = \{ \overline{1}, \overline{3}, \overline{11} \}\), and \( |Y| < |R| - 2\).  Now, by Theorem~\ref{Estimate right partition}, the right partition number \( \rho \) of \( (A,B) \) satisfies
\[
\rho \geq \left\lceil \frac{|R|}{|Y|} \right\rceil = \left\lceil \frac{5}{2} \right\rceil = 3.
\]  
In fact, \( \rho = 3 \).  To see this, observe that $B$ is the union of the (pairwise disjoint) ranges of the following three partial matchings: 
\begin{align*}
f &: \overline{0}\mapsto \overline{3}, \, \overline{1}\mapsto \overline{2}, \, \overline{2}\mapsto \overline{1}, \, \overline{10}\mapsto \overline{11}, \, \overline{11}\mapsto \overline{10} \\
g &: \overline{1}\mapsto \overline{6}, \, \overline{11}\mapsto \overline{4} \\
h &: \overline{1}\mapsto \overline{8}.
\end{align*}
Note that the above partition of \( B \) into three right-admissible sets is not unique. 
\end{example}

Next we turn to the decomposition of \( A \) into left-admissible sets. As usual, we assume that \( A \) and \( B \) are nonempty finite subsets of an abelian group, with \( |A| = |B| \) and \( 1 \notin B \). These conditions guarantee that \( xB \not\subseteq A \) for each \( x \in A \).  Hence, for all nonempty \( S \subseteq  A \), the set \( U_S = \{ b \in B : Sb \subseteq A \} \) cannot equal \( B \).  With this in mind, we define the {\em left partition number} of the pair \( (A,B) \) to be the maximum value of 
\[
\left\lceil \frac{|S|}{|B| - |U_S|} \right\rceil,
\]
where \( S \) ranges over all nonempty subsets of \( A \).  We denote the left partition number by \( \lambda = \lambda(A,B) \).

\begin{theorem}\label{lpn}
Let \( A \) and \( B \) be nonempty finite subsets of an abelian group \( G \), with \( |A| = |B| \) and \( 1 \notin B \).  Let \( \lambda \) be the left partition number of \( (A,B) \). Then \( A \) can be written as a union of \( \lambda \), and no fewer, pairwise disjoint left-admissible subsets. 
\end{theorem}

\begin{proof}
Let \( S \) be a nonempty subset of \( A \), and let \(  U_S = \{ b \in B : Sb \subseteq A \} \).  Recall that \( B \setminus U_S = \Delta(S) \) and \( U_S \neq B \).  We compute 
\begin{align*}
|S| = \frac{|S|}{|B| - |U_S|}\left(|B| - |U_S|\right) &= \frac{|S|}{|B| - |U_S|}|B \setminus U_S| \\
&= \frac{|S|}{|B| - |U_S|}|\Delta(S)| \leq \lambda |\Delta(S)|.
\end{align*}
Thus \( |S| \leq \lambda |\Delta(S)| \), and observe that this inequality also holds when \( S \) is the empty set.  Therefore, by Theorem~\ref{Partition A}, \( A \) can be written as a union of \( \lambda \) pairwise disjoint left-admissible sets.

For the minimality assertion, suppose that \( A \) can be partitioned into \( k \) left-admissible sets.  Then, by Theorem~\ref{Partition A}, for each nonempty subset \( S \) of \( A \), 
\[ 
|S| \leq  k| \Delta(S)|. 
\]
Again using the fact that \( \Delta(S) = B \setminus U_S \), we rewrite this inequality as 
\[
|S| \leq k(|B| - |U_S|).
\]
Since \( U_S \subsetneq B \), we can divide by \( |B| - |U_S| \), to obtain   
\[
k \geq \frac{|S|}{|B| - |U_S|}.
\]
Taking the ceiling of both sides, we find, by the definition of \( \lambda \), that \( k \geq \lambda \). 
\end{proof}

One would like to have a result similar to Theorem~\ref{Formula for right partition} for the left partition number.  Note, however, that Theorem~\ref{Formula for right partition} depends on Theorem~\ref{Equivalence for groups}, and the proof given for (ii) implies (i) in the latter theorem does not extend to the case where \( \alpha < 1 \). A new idea will likely be needed, but we can at least provide the following lower bound for \( \lambda \).

\begin{proposition}
Let \( A \) and \( B \) be nonempty finite subsets of an abelian group \( G \), with \( |A| = |B| \) and \( 1 \notin B \).  Let \( \lambda \) be the left partition number of \( (A,B) \). Then \( \lambda \) is bounded below by the maximum value of 
\[
\left\lceil \frac{|S|}{|B| - |R\setminus\{ 1 \}|} \right\rceil,
\]
where \( S \) and \( R \) range over all nonempty subsets of \( A \) and \( B \cup \{ 1 \} \), respectively, such that \( SR = S \).
\end{proposition}

\begin{proof}
Observe that for any pair of sets \( S, R \) as in the statement, we have
\[
R \setminus \{ 1 \} \subseteq U_S \subsetneq B,
\]
hence
\[
\left\lceil \frac{|S|}{|B| - |U_S|} \right\rceil \geq \left\lceil \frac{|S|}{|B| - |R \setminus \{ 1 \}|} \right\rceil.
\]
The conclusion now follows from the definition of $\lambda$.
\end{proof}

\section{Concluding remarks and future directions}\label{FutureDirections}
\begin{enumerate}
\item Extending the partial matching theory in this paper to the non-abelian setting would be a natural next step.  This will likely require identifying the correct noncommutative replacements for the stabilizing condition
\(SR=S\) and for the coset-union structural statements that underlie our obstruction theory.  One expects that left/right stabilizers and two-sided product-set growth (in the spirit of Kemperman-type transforms and isoperimetric methods) will play a central role. 

\item While \(\rho(A,B)\) admits a sharp description in terms of defect, the parallel theory for the left partition number \(\lambda(A,B)\) remains incomplete. It would be particularly interesting to obtain a characterization of \(\lambda(A,B)\) by means of a structural theorem or a defect-type criterion.  New ideas seem necessary here, as the techniques used for \(\rho(A,B)\) do not directly transfer.

\item Theorems~\ref{Structure of PMG} and \ref{Existence partial} motivate extremal questions:
for fixed \(|G|\), \(|A|=|B|=n\), and defect \(d\), what is the maximum number of pairs \((A,B)\)
with \(\delta(A,B)=d\), and what configurations asymptotically attain this?
In cyclic groups one can ask for sharp counts or sharp asymptotics of pairs with a given defect.
\end{enumerate}

\medskip

\noindent \textbf{Declarations} 

\medskip

\noindent \textbf{\small Conflict of interest:} {\small The authors declare that they have no conflict of interest.}
\noindent \textbf{\small Data availability statement:} {\small Data sharing is not applicable to this article as no datasets were generated or analyzed during the current study.}

\enlargethispage{\baselineskip}

\end{document}